\documentclass[a4paper]{amsart}

\usepackage{latexsym,amssymb,amsmath,graphics}
\usepackage[dvips]{graphicx}

\def\cc{{\mathcal C}}
\def\dd{{\mathcal D}}
\def\ee{{\mathcal E}}
\def\ff{{\mathcal F}}

\def\kk{{\mathcal K}}

\def\mm{{\mathcal M}}
\def\nn{{\mathcal N}}
\def\pp{{\mathcal P}}

\def\ss{{\mathcal S}}
\def\tt{{\mathcal T}}

\def\ux{\underline{x}}
\def\uy{\underline{y}}
\def\ie{i.e.\ }

\def\eps{\varepsilon}
\def\dst{\displaystyle}

\def\supp{{\mathrm{supp}\,}}
\def\span{{\mathrm{span}\,}}

\renewcommand{\Im}{\mathrm{Im}\,}

\def\C{{\mathbb{C}}}

\def\N{{\mathbb{N}}}

\def\R{{\mathbb{R}}}
\def\S{{\mathbb{S}}}

\def\d{\,{\mathrm{d}}}
\newcommand{\norm}[1]{{\left\|{#1}\right\|}}

\newcommand{\abs}[1]{{\left|{#1}\right|}}
\newcommand{\scal}[1]{{\left\langle{#1}\right\rangle}}

\newenvironment{definition}[1][]{\vskip3pt\noindent\sl\textbf{Definition.}\ }{\rm\vskip3pt}

\newtheorem{lemma}{Lemma}[section]
\newtheorem{proposition}[lemma]{Proposition}
\newtheorem{theorem}[lemma]{Theorem}
\newtheorem{corollary}[lemma]{Corollary}
\newtheorem{remarknum}[lemma]{Remark}

\date{\today}

\begin{document}
\title{Uncertainty principles for integral operators}
\author{Saifallah Ghobber}

\address{S. G\,: D\'epartement de Math\'ematiques Appliqu\'{e}es\\
Institut Pr\'{e}paratoire Aux  \'{E}tudes D'ing\'{e}nieurs de Nabeul\\
Universit\'e de Carthage\\
Campus Universitaire, Merazka\\ 8000, Nabeul\\
Tunisie}
\email{Saifallah.Ghobber@math.cnrs.fr}

\author{Philippe Jaming}

\address{P. J\,: Univ. Bordeaux, IMB, UMR 5251, F-33400 Talence, France.
CNRS, IMB, UMR 5251, F-33400 Talence, France.}
\email{Philippe.Jaming@u-bordeaux1.fr}

\begin{abstract}
The aim of this paper is to prove new uncertainty principles for an integral operator $\tt$ with a bounded kernel for which there is a Plancherel theorem.
The first of these results is an extension of  Faris's local uncertainty principle which states that if a nonzero function  $f\in L^2(\R^d,\mu)$ is highly
localized near a single point then  $\tt (f)$ cannot be concentrated  in  a set of finite measure.
The second result extends the Benedicks-Amrein-Berthier uncertainty principle and states that a nonzero function
$f\in L^2(\R^d,\mu)$ and its integral transform $\tt (f)$ cannot both have support of finite measure.
From these two results we deduce a global uncertainty principle of Heisenberg type for the transformation $\tt$.
We apply our results to obtain a new uncertainty principles for the Dunkl and Clifford Fourier transforms.
\end{abstract}



\maketitle


\section{Introduction}
Uncertainty principles are mathematical results that give limitations on the simultaneous concentration of a function
and its Fourier  transform. They  have implications in two main areas: quantum physics and signal analysis. In
quantum physics they tell us that a particle's speed and position cannot both be measured
with infinite precision. In signal analysis they tell us that if we observe a signal only for a
finite period of time, we will lose information about the frequencies the signal consists of.
There are many ways to get the statement about concentration precise. The most famous of them is the so called
Heisenberg Uncertainty Principle \cite{Heis}
where concentration is measured by dispersion and the Hardy Uncertainty Principle \cite{hardy}  where concentration is
measured in terms of fast decay. A little less known one consists in measuring concentration in terms of smallness of
support.
A considerable attention has been devoted recently to discovering new formulations and new contexts for the
uncertainty principle
({\it see}   the surveys  \cite{BD, folland} and the book \cite{HJ} for other forms of the uncertainty principle).

Our aim here is to consider uncertainty principles in which concentration is measured either
by (generalized) dispersion like in Heisenberg's uncertainty principle
or by the smallness of the support.
The transforms under consideration are
integral operators $\tt$ with polynomially bounded kernel $\kk$ and for which there is
a Plancherel Theorem
and include the usual Fourier transform, the Fourier-Bessel (Hankel) transform, the Fourier-Dunkl
transform and the Fourier-Clifford transform as particular cases.

\medskip

Let us now be more precise. Let $\Omega$, $\widehat{\Omega}$ be two convex cones in $\R^d$ ({\it i.e.}
$\lambda x\in\Omega$ if $\lambda>0$ and $x\in\Omega$) with non-empty interior. We endow them with  Borel measures
$\mu$ and $\widehat{\mu}$. The Lebesgue spaces $L^p(\Omega,\mu)$, $1\leq p \leq \infty$,
are then defined in the usual way.
We assume that the measure $\mu$ is absolutely continuous with respect to the Lebesgue measure and has a polar
decomposition of the form $\mbox{d}\mu(r\zeta)=r^{2a-1}\,\mbox{d}r\,Q(\zeta)\,\mbox{d}\sigma(\zeta)$
where $\mbox{d}\sigma$ is the Lebesgue measure on the unit sphere $\S^{d-1}$ of $\R^d$ and 
$Q\in L^1(\S^{d-1},\mbox{d}\sigma)$, $Q\not=0$.
Then $\mu$ is homogeneous of degree $2a$ in the following sense:
for every continuous function $f$ with compact support in $\Omega$ and every $\lambda>0$,
\begin{equation}
\label{eqmesure}
\int_{\Omega} f\left(\frac{x}{\lambda}\right)\d \mu(x)= \lambda^{2a}\int_{\Omega} f(x)\d \mu(x).
\end{equation}
We define $\widehat{a}$ accordingly for $\widehat{\mu}$ and assume that $\widehat{a}=a$.

Next, let $\kk:\Omega\times \widehat{\Omega}\longrightarrow \C$ be a kernel such that
\begin{enumerate}
	\item $\kk$ is continuous;
  \item $\kk$ is polynomially bounded: $|\kk(x,\xi)|\le c_{\tt}(1+|x|)^m(1+|\xi|)^{\widehat{m}}$;
  \item $\kk$ is homogeneous: $\kk(\lambda x,\xi)= \kk(x,\lambda \xi)$.
\end{enumerate}
One can then define the integral operator
$\tt$ on $\ss(\Omega)$ by
\begin{equation}\label{defop}
\tt(f)(\xi)=\int_{\Omega}f(x)\kk(x,\xi)\d \mu(x), \; \xi\in \widehat{\Omega}.
\end{equation}
For $\rho>0$, we define the measures $\d \mu_\rho(x)=(1+|x|)^\rho\d\mu(x)$ and
$\d \widehat{\mu}_\rho(\xi)=(1+|\xi|)^\rho\d\widehat{\mu}(\xi)$. Then $\tt$ extends into
a continuous operator from $L^1(\Omega,\mu_m)$ to
$$
\cc_{\widehat{m}}(\widehat{\Omega})=\left\{f\mbox{ continuous s.t.}  \norm{f}_{\infty,\widehat{m}} :=
\sup_{\xi\in\widehat{\Omega}}\frac{|f(\xi)|}{(1+|\xi|)^{ \widehat{m}}}<\infty\right\}.
$$

Further, if we introduce the dilation operators $\dd_{\lambda},\widehat{\dd}_{\lambda}$, $\lambda>0$:
$$
\dd_{\lambda} f(x) = \frac{1}{\lambda^{a}} f\left(\frac{x}{\lambda}\right), \quad
\widehat{\dd}_{\lambda} f(x) = \frac{1}{\lambda^{\widehat{a}}} f\left(\frac{x}{\lambda}\right),
$$
then the homogeneity of $\kk$ implies
\begin{equation}\label{dt}
\tt \dd_{\lambda}= \widehat{\dd}_{\frac{1}{\lambda}}\tt.
\end{equation}
Also, from the fact that $\mu,\widehat{\mu}$ are absolutely continuous with respect to the Lebesgue measure, these
dilation operators are continuous from $(0,\infty)\times L^2(\Omega,\mu_\rho)$ --resp.  $(0,\infty)\times
L^2(\widehat{\Omega},\widehat{\mu}_\rho)$-- to $L^2(\Omega,\mu_\rho)$ --resp.  $
L^2(\widehat{\Omega},\widehat{\mu}_\rho)$.

The integral operators under consideration will be assumed to satisfy some of
the following proprieties that are common for Fourier-like transforms:
\begin{enumerate}
\item $\tt$ has an {\bf Inversion Formula}: {\sl When both $f \in L^1(\Omega,\mu_m)$ and
$\tt(f)\in L^1(\widehat{\Omega},\widehat{\mu}_{\widehat{m}})$ we have $f\in \cc_m(\Omega)$ and}
$$
f(x)=\tt^{-1}[\tt(f)](x)=\int_{\widehat{\Omega}} \tt(f)(\xi)\overline{\kk(x,\xi)}\d \widehat{\mu}(\xi), \; x\in \Omega.
$$

\item $\tt$ satisfies {\bf Plancherel's Theorem}: {\sl for every $f\in\ss(\Omega)$, $\norm{\tt(f)}_{L^2(\widehat{\Omega},\widehat{\mu})}=\norm{f}_{L^2(\Omega,\mu)}$.
In particular, $\tt$ extends uniquely to a unitary transform from $L^2(\Omega,\mu)$ onto $L^2(\widehat{\Omega},\widehat{\mu})$.}\footnote{This condition implies $\widehat{a}=a$.}


\end{enumerate}

This family of transforms include for instance the Fourier transform and the Fourier-Dunkl transform. We will also slightly relax the conditions to include the Fourier-Clifford transform.
We will here concentrate on uncertainty principles where concentration is measured in terms of dispersion or in terms
of smallness of support. Our first result will be the following local uncertainty principle
that we state here in the case $m=\widehat{m}=0$ for simplicity:

\medskip

\noindent{\bf Theorem A.}\\
{\sl Assume $m=\widehat{m}=0$.
Let $\Sigma \subset \widehat{\Omega}$ be a measurable subset of finite measure $0<\widehat{\mu}(\Sigma)<\infty$.
Then
\begin{enumerate}
  \item if $\,0 <s <a$, there is a constant $C$ such that for all $f\in L^2(\Omega,\mu)$,
 \begin{equation*}
\norm{\tt(f)}_{L^2(\Sigma,\widehat{\mu})} \leq
C\Big[\widehat{\mu}(\Sigma)\Big]^{\frac{s}{2a}}\big\|\abs{x}^s f\big\|_{L^2(\Omega,\mu)};
\end{equation*}
\item if $\,s >a$, there is a constant $C$ such that for all $f\in L^2(\Omega,\mu)$,
\begin{equation*}
\norm{\tt(f)}_{L^2(\Sigma,\widehat{\mu})} \leq
C\Big[\widehat{\mu}(\Sigma)\Big]^{\frac{1}{2}}\big\|f\big\|_{L^2(\Omega,\mu)}^{1-\frac{a}{s}}
\big\|\abs{x}^s f\big\|_{L^2(\Omega,\mu)}^{\frac{a}{s}}.
\end{equation*}
\end{enumerate}
}

\medskip

This theorem implies that if $f$ is highly localized in the  neighborhood of $0$,
 \ie  the dispersion  $\big\|\abs{x}^s f\big\|_{L^2(\Omega,\mu)}$ takes a small  value, then  $\tt(f)$ cannot be concentrated in a subset $\Sigma$ of finite measure.
We can refer to \cite{Fa, Pr, Prr, Prrr} for the history of these uncertainty inequalities.

Another uncertainty principle which is of particular interest is:
{\sl a function $f$  and its integral transform $\tt(f)$ cannot both have small support}.
 In other words we are interested in the following adaptation of a well-known notion from Fourier analysis:

\begin{definition}\ \\
\label{num}
Let $S\subset\Omega$, $\Sigma\subset\widehat{\Omega}$ be two measurable subsets. Then
\begin{itemize}
\item $(S,\Sigma)$ is a \emph{weak annihilating pair}\footnote{see also the very similar notion of Heisenberg uniqueness pairs \cite{HMR}.} if, $\supp f \subset S$  and $\supp \tt(f)\subset \Sigma$ implies $f=0$.
\item $(S,\Sigma)$ is called a \emph{strong annihilating pair} if there exists $C=C(S,\Sigma)$ such that for every $f\in L^2(\Omega,\mu)$
\begin{equation}\label{int,str}
\norm{f}^2_{L^2(\Omega,\mu)}\leq C \Big(\norm{f}^2_{L^2(S^c,\mu)} +\norm{\tt(f)}^2_{L^2(\Sigma^c,\widehat{\mu})} \Big),
\end{equation}
\end{itemize}
where $A^c$ is the complementary of the set $A$ in $\Omega$ or $\widehat{\Omega}$. The constant $C(S,\Sigma)$  will be called the annihilation constant of $(S,\Sigma)$.
\end{definition}
Of course, every strong annihilating pair is also a weak one.
To prove that a pair $(S,\Sigma)$ is a strong annihilating pair, it is enough to shows that there exists a constant 
$D(S,\Sigma)$ such that for all $f \in L^2(\Omega,\mu)$ supported in $S$
\begin{equation}\label{defst2}
\norm{f}^2_{L^2(\Omega,\mu)}\leq D(S,\Sigma)\norm{\tt(f)}^2_{L^2(\Sigma^c,\widehat{\mu})}.
\end{equation}

 The qualitative (or weak) uncertainty principle has been considered in various places
\cite{al, Be, cps, ekk, ho, kan, k, ps}.
Our main concern here is the quantitative (or strong) uncertainty principles of the form  \eqref{int,str}.
In his paper \cite{dj}, de Jeu proved a
quite general uncertainty principle for integral operators with bounded transform. This result states that if $S,\Sigma$ are sets with sufficiently small measure,
 then $(S,\Sigma)$ is a strong annihilating pair.
One is thus lead to ask whether any pair of sets of finite measure is strongly annihilating.

In the case of the Fourier transform, this was proved
by Amrein-Berthier \cite{AB} (while the weak counter-part was proved by Benedicks \cite{Be}).
It is interesting to note that, when $f\in L^2(\R^d)$ the optimal estimate of $C$,
which depends only on Lebesgue's measures $|S|$ and $|\Sigma|$, was obtained by F. Nazarov \cite{Na} ($d=1$), while in
higher dimension the question is not fully settled unless either $S$ or $\Sigma$
is convex ({\it see}  \cite{JA} for the best result today). For the Fourier-Bessel/Hankel transform, this was done by
the authors in \cite{GJ}.
Our main result will be the following adaptation of the Benedicks-Amrein-Berthier uncertainty principle:

\medskip

\noindent{\bf Theorem B.}\\
{\sl Let $S\subset\Omega$, $\Sigma\subset\widehat{\Omega}$ be a pair of measurable subsets with $0<\mu_{2m}(S),\widehat{\mu}_{2\widehat{m}}(\Sigma)<\infty$.
Then there exists a constant $C(S, \Sigma)$ such that for any function $f\in L^2(\Omega,\mu)$,
\begin{equation*}
\norm{f}^2_{L^2(\Omega,\mu)}\leq C(S, \Sigma)\Big(\norm{f}^2_{L^2(S^c,\mu)}
 +\norm{\tt(f)}^2_{L^2(\Sigma^c,\widehat{\mu})} \Big).
\end{equation*}
}
For the Fourier transform the proof of this theorem  in  stated in \cite{AB} where the translation and the modulation operators plays a key role. Our
theorem include essentially  integral operators for which the translation operator is not explicit (the Dunkl transform for example) or does not behave like
 the ordinary translation (the Fourier-Bessel transform for example).
To do so we will replace translation by dilation and use the fact that the dilates of a $\cc_0$-function are linearly independent ({\it see} Lemma \ref{libre}).

Finally, from either Theorem $A$ or Theorem $B$ we will deduce the following global uncertainty inequality:

\medskip

\noindent{\bf Theorem C.}\\
{\sl For $s,\;\beta>0$, there exists a constant $C_{s,\beta}$ such that for all $f\in L^2(\Omega,\mu)$,
$$
\big\||x|^{s}f\big\|^{\frac{2\beta}{s+\beta}}_{L^2(\Omega,\mu)} \; \big\||\xi|^{\beta}\tt(f)\big\|^{\frac{2s}{s+\beta}}_{L^2(\widehat{\Omega},\widehat{\mu})}
\ge C_{s,\beta} \norm{f}^{2}_{L^2(\Omega,\mu)}.
$$
}

In particular when $s=\beta=1$ we obtain a Heisenberg uncertainty principle type for the transformation $\tt$.

\medskip

The structure of the paper is as follows: in the next section we will prove the local uncertainty inequality for the transformation $\tt$.
Section $3$ is devoted to our Benedicks-Amrein-Berthier type theorem and in Section $4$ we apply our results
for the Dunkl  and the Clifford Fourier transforms.

\medskip

\subsection*{Notation} Throughout this paper we denote by $\scal{.,.}$
the usual Euclidean inner product in $\R^d$, we write for $x \in \R^d$,
$\abs{x}=\sqrt{\scal{x,x}}$ and if $S$ is a measurable subset in $\R^d$, we will write $|S|$ for its Lebesgue measure.

Finally, $\S^{d-1}$ is the unit sphere on $\R^d$ endowed with the normalized surface measure $\d \sigma$.

We will write $c(\tt)$ (resp. $c(s,\tt)$...) for a constant that depends on the parameters $a,m,\hat m$ and $c_\tt$
defined above (resp. to indicate the dependence on some other parameter $s$...).
This constants may change from line to line.

\section{Local uncertainty principle}
Local uncertainty inequalities for the Fourier transform were firstly obtained by Faris \cite{Fa}, and they were subsequently sharpened and generalized
by Price and Sitaram \cite{Pr, Prr}. Similar inequalities on Lie groups of polynomial growth were established by
Ciatti, Ricci and Sundari in \cite{crs} which is based on \cite{Prrr} and further extended in \cite{Mar}

First from the polar decomposition of our measure we remark that
\begin{equation}\label{eqlocfini}
 \begin{cases}
\dst C_1(s):=\int_{\Omega\cap\{|x|\leq 1\}}\frac{\d\mu(x)}{|x|^{2s}}<\infty, & 0<s<a;\\
\dst C_2(s):=\int_{\Omega}\frac{\d\mu(x)}{(1+|x|)^{2s}}<\infty, &  s>a.
\end{cases}
 \end{equation}

\begin{theorem}\ \\
Let $\Sigma \subset \widehat{\Omega}$ be a measurable subset of finite measure $0<\widehat{\mu}_{2\widehat{m}}(\Sigma)<\infty$.
Then
\label{prop:local.}
\begin{enumerate}
  \item if $\,0 <s <a$, there is a constant $c(s,\tt)$ such that for all $f\in L^2(\Omega,\mu)$,
 \begin{equation}\label{eq1}
\norm{\tt(f)}_{L^2(\Sigma,\widehat{\mu})} \leq
\begin{cases} c(s,\tt)\Big[\widehat{\mu}_{2\widehat{m}}(\Sigma)\Big]^{\frac{s}{2(a+m)}}\big\|\abs{x}^s f\big\|_{L^2(\Omega,\mu)},&\mbox{if }\widehat{\mu}_{2\widehat{m}}(\Sigma)\leq 1;\\
c(s,\tt)\Big[\widehat{\mu}_{2\widehat{m}}(\Sigma)\Big]^{\frac{s}{2a}}\big\|\abs{x}^s f\big\|_{L^2(\Omega,\mu)},&\mbox{if }\widehat{\mu}_{2\widehat{m}}(\Sigma)> 1;
\end{cases}
\end{equation}
\item if $\,a\leq s\leq a+m$ then, for every $\eps>0$ there is a constant $c(s,\tt,\eps)$
such that for all $f\in L^2(\Omega,\mu)$,
$$
\norm{\tt(f)}_{L^2(\Sigma,\widehat{\mu})}
\leq
\begin{cases}
c(s,\tt,\eps)\Big[\widehat{\mu}_{2\widehat{m}}(\Sigma)\Big]^{\frac{1}{2(1+m/a)}-\eps}
\big\|f\big\|_{L^2(\Omega,\mu)}^{1-\frac{a}{s}+\eps}
\big\|\abs{x}^s f\big\|_{L^2(\Omega,\mu)}^{\frac{a}{s}-\eps},
&\mbox{if }\widehat{\mu}_{2\widehat{m}}(\Sigma)\le 1;\\ c(s,\tt,\eps)\Big[\widehat{\mu}_{2\widehat{m}}(\Sigma)\Big]^{\frac{1}{2}-\eps}
\big\|f\big\|_{L^2(\Omega,\mu)}^{1-\frac{a}{s}+\eps}
\big\|\abs{x}^s f\big\|_{L^2(\Omega,\mu)}^{\frac{a}{s}-\eps}, &\mbox{if }\widehat{\mu}_{2\widehat{m}}(\Sigma)> 1;
\end{cases}
$$
\item if $\,s >m+ a$, there is a constant $c(s,\tt)$ such that for all $f\in L^2(\Omega,\mu)$,
\begin{equation}\label{eq2}
\norm{\tt(f)}_{L^2(\Sigma,\widehat{\mu})} \leq
\begin{cases} c(s,\tt)\Big[\widehat{\mu}_{2\widehat{m}}(\Sigma)\Big]^{\frac{1}{2}}\big\|f\big\|_{L^2(\Omega,\mu)}^{1-\frac{a}{s}}
\big\|\abs{x}^s f\big\|_{L^2(\Omega,\mu)}^{\frac{a}{s}},&\mbox{if }m=0;\\
c(s,\tt)\Big[\widehat{\mu}_{2\widehat{m}}(\Sigma)\Big]^{\frac{1}{2}}\big\|f\big\|_{L^2(\Omega,\mu_{2s})},&\mbox{otherwise}.\end{cases}
\end{equation}
\end{enumerate}
\end{theorem}

\begin{proof}
As for the first part take $r>0$ and
let $\chi_r= \chi_{\Omega\cap\{|x|\leq r\}}$  and $\tilde{\chi_r}=1- \chi_r$. We may then write
$$
\norm{\tt(f)}_{L^2(\Sigma,\widehat{\mu})} =\norm{\tt(f)\chi_\Sigma}_{L^2(\widehat{\Omega},\widehat{\mu})}
\leq \norm{\tt(f\chi_r)\chi_\Sigma}_{L^2(\widehat{\Omega},\widehat{\mu})}+\norm{\tt(f\tilde{\chi_r})}_{L^2(\widehat{\Omega},\widehat{\mu})},
$$
hence, it follows from Plancherel's theorem  that
$$
\norm{\tt(f)}_{L^2(\Sigma,\widehat{\mu})}
\leq \widehat{\mu}_{2\widehat{m}}(\Sigma)^{1/2}\norm{\tt(f\chi_r)}_{\infty,\widehat{m}} + \norm{f\tilde{\chi_r}}_{L^2(\Omega,\mu)}.
$$
Now  we have
\begin{eqnarray*}
\norm{\tt(f \chi_r)}_{\infty,\widehat{m}}&\leq& c_{\tt}\norm{f\chi_r}_{L^1(\Omega,\mu_m)}
\leq c_{\tt}\bigl\|\abs{x}^{-s}(1+|x|)^m\chi_r\bigr\|_{L^2(\Omega,\mu)} \bigl\|\abs{x}^{s}f\bigl\|_{L^2(\Omega,\mu)}\\
&\le&c_{\tt} \sqrt{C_1(s)}\, (1+r)^mr^{a-s}\big\|\abs{x}^{s}f\big\|_{L^2(\Omega,\mu)}.
\end{eqnarray*}
On the other hand,
$$
\norm{f\tilde{\chi_r}}_{L^2(\Omega,\mu)}
\leq \big\|\abs{x} ^{-s}\tilde{\chi_r}\big\|_{L^\infty(\Omega,\mu)} \big\|\abs{x}^{s}f\big\|_{L^2(\Omega,\mu)}
= r^{-s}\big\|\abs{x}^{s}f\big\|_{L^2(\Omega,\mu)},
$$
so that
$$
\norm{\tt(f)}_{L^2(\Sigma,\widehat{\mu})}
\leq \Big( r^{-s}+ 2c_{\tt} \sqrt{C_1(s)}\, (1+r)^mr^{a-s}\widehat{\mu}_{2\widehat{m}}(\Sigma)^{1/2}\Big)\big\|\abs{x}^{s}f\big\|_{L^2(\Omega,\mu)}.
$$
If $\widehat{\mu}_{2\widehat{ m}}(\Sigma)>1$ we take $r=\widehat{\mu}_{2 \widehat{ m}}(\Sigma)^{-1/2a}<1$ (thus $(1+r)^m\leq 2^m$)
to obtain that there is a constant $C$ depending only on $s$ and $\tt$ such that
$$
\norm{\tt(f)}_{L^2(\Sigma,\widehat{\mu})}
\leq C\widehat{\mu}_{2\widehat{m}}(\Sigma)^{s/2a}
\big\|\abs{x}^{s}f\big\|_{L^2(\Omega,\mu)}.
$$
If $\widehat{\mu}_{2\widehat{ m}}(\Sigma)<1$ we take $r=\widehat{\mu}_{2\widehat{ m}}(\Sigma)^{-1/2(a+m)}>1$ (thus $(1+r)^m\leq 2^mr^m$)
to obtain that there is a constant $C$ depending only on $s$ and $\tt$ such that
$$
\norm{\tt(f)}_{L^2(\Sigma,\widehat{\mu})}
\leq C\widehat{\mu}_{2\widehat{m}}(\Sigma)^{s/2(a+m)}
\big\|\abs{x}^{s}f\big\|_{L^2(\Omega,\mu)}.
$$

\smallskip

Next, take $0<\sigma<a\leq s\leq a+m$, apply the first part with $\sigma$ replacing $s$ and then apply the classical
inequality
 \begin{equation*}\label{eq:classineq}
\norm{|x|^\sigma f}_{L^2(\Omega,\mu)}\leq C(\sigma,s)\norm{f}_{L^2(\Omega,\mu)}^{1-\frac{\sigma}{s}}
\norm{|x|^s f}_{L^2(\Omega,\mu)}^{\frac{\sigma}{s}}.
\end{equation*}

\smallskip

As for the last part we write
$$
\norm{\tt(f)}_{L^2(\Sigma,\widehat{\mu})}\leq\widehat{\mu}_{2\widehat{m}}(\Sigma)^{1/2}
\norm{\tt(f)}_{\infty,\widehat{m}}
\leq c_{\tt} \widehat{\mu}_{2\widehat{m}}(\Sigma)^{1/2}\norm{f}_{L^1(\Omega,\mu_m)}.
$$
Moreover
\begin{eqnarray*}
\norm{f}^2_{L^1(\Omega,\mu_m)}&=&
 \left(\int_{\Omega} (1+\abs{x})^{m}
|f(x)|\d \mu(x)\right)^2 \\&=&
 \left(\int_{\Omega} (1+\abs{x})^{-(s-m)}(1+\abs{x})^{s}
|f(x)|\d \mu(x)\right)^2  \\
&\leq&C_2(s-m)\int_{\Omega}(1+\abs{x})^{2s}|f(x)|^2\d \mu(x).
\end{eqnarray*}
Further, if $m=0$, then this last inequality implies
$$
\norm{f}^2_{L^1(\Omega,\mu)}\leq 2^{2s}C_2(s)\Big( \norm{f}^2_{L^2(\Omega,\mu)}+\big\|\abs{x}^{s}f\big\|^2_{L^2(\Omega,\mu)} \Big).
$$
Replacing $f$ by $ \dd_\lambda f$, $\lambda>0$ in this inequality, gives
$$
\norm{f}^2_{L^1(\Omega,\mu)} \leq 2^{2s}C_2(s)
\left(\lambda^{-2a} \big\|f\big\|^2_{L^2(\Omega,\mu)}+\lambda^{2(s-a)}\big\|\abs{x}^{s}f\big\|^2_{L^2(\Omega,\mu)} \right).
$$
Minimizing the right hand side of that inequality over $\lambda > 0$, we obtain the desired result.
\end{proof}
%

We now show that local uncertainty principle implies a global uncertainty principle type for $\tt$.
For sake of simplicity, we will assume that $m=\widehat{m}=0$. The general case will be treated in the next section.

\begin{corollary}\ \\
Assume that $m=\widehat{m}=0$.
For $s,\;\beta>0$, $s\not=a$ there exists a constant $C= C(s,\beta,\tt)$ such that for all $f\in L^2(\Omega,\mu)$,
\begin{equation}
\label{eqhe.}
\big\||x|^{s}f\big\|^{\frac{\beta}{s+\beta}}_{L^2(\Omega,\mu)} \;
\big\||\xi|^{\beta}\tt(f)\big\|^{\frac{s}{s+\beta}}_{L^2(\widehat{\Omega},\widehat{\mu})}
\ge C \norm{f}_{L^2(\Omega,\mu)}.
\end{equation}
\end{corollary}

\begin{proof}
In this proof, we will denote by $B_r=\widehat{\Omega}\cap\{x\,: |x|\leq r\}$ and
$B_r^c=\widehat{\Omega}\setminus B_r$.

Let $0<s<a$ and $\beta>0$. Then, using Plancherel's theorem and Theorem \ref{prop:local.} (1),
\begin{eqnarray*}
\norm{f}^{2}_{L^2(\Omega,\mu)} &=& \norm{\tt(f)}^{2}_{L^2(\widehat{\Omega},\widehat{\mu})}=
 \norm{\tt(f)}^{2}_{L^2(B_r,\widehat{\mu})}+\norm{\tt(f)}^{2}_{L^2(B_r^c,\widehat{\mu})} \nonumber\\
&\le& c(s,\tt)\widehat{\mu}(B_r)^{\frac{s}{a}}\bigl\|\abs{x}^{s}f\bigr\|^{2}_{L^2(\Omega,\mu)}
 + r^{-2\beta}\bigl\||\xi|^{\beta}\tt(f)\bigr\|^{2}_{L^2(\widehat{\Omega},\widehat{\mu})}\nonumber\\
&\le& c'(s,\tt)r^{2s}\bigl\|\abs{x}^{s}f\bigr\|^{2}_{L^2(\Omega,\mu)}
+ r^{-2\beta}\bigl\||\xi|^{\beta}\tt(f)\bigr\|^{2}_{L^2(\widehat{\Omega},\widehat{\mu})}.
\end{eqnarray*}
The desired result follows by minimizing the right hand side of that inequality over $r > 0$.

For $s>a$ and $\beta>0$ we deduce from Plancherel's theorem and Theorem \ref{prop:local.} (3) that
\begin{eqnarray}
\norm{f}^{2}_{L^2(\Omega,\mu)} &=&
\norm{\tt(f)}^{2}_{L^2(\widehat{\Omega},\widehat{\mu})}
= \norm{\tt(f)}^{2}_{L^2(B_r,\widehat{\mu})}+\norm{\tt(f)}^{2}_{L^2(B_r^c,\widehat{\mu})} \nonumber\\
&\le& c(s,\tt)^2\bigl\|f\bigr\|^{2-\frac{2a}{s}}_{L^2(\Omega,\mu)}
\widehat{\mu}(B_r)\bigl\|\abs{x}^{s}f\bigr\|_{L^2(\Omega,\mu)}^{\frac{2a}{s}}
+ \norm{\tt(f)}^{2}_{L^2(B_r^c,\widehat{\mu})}\label{eq:farisaaa}.
\end{eqnarray}
But, using Plancherel's theorem again,
$$
\norm{\tt(f)}^{2}_{L^2(B_r^c,\widehat{\mu})}
\leq \norm{\tt(f)}_{L^2(B_r^c,\widehat{\mu})}^{\frac{2a}{s}}
\norm{\tt(f)}_{L^2(\widehat{\Omega},\widehat{\mu})}^{2-\frac{2a}{s}}
=\norm{\tt(f)}_{L^2(B_r^c,\widehat{\mu})}^{\frac{2a}{s}}
\norm{f}_{L^2(\Omega,\mu)}^{2-\frac{2a}{s}}
$$
so that, in \eqref{eq:farisaaa}, we may simplify by $\norm{f}_{L^2(\Omega,\mu)}^{2-\frac{2a}{s}}$
to obtain
\begin{eqnarray*}
\norm{f}^{\frac{2a}{s}}_{L^2(\Omega,\mu)}&\leq& c(s,\tt)^2
 \widehat{\mu}(B_r)\bigl\|\abs{x}^{s}f\bigr\|_{L^2(\Omega,\mu)}^{\frac{2a}{s}}
+\norm{\tt(f)}_{L^2(B_r^c,\widehat{\mu})}^{\frac{2a}{s}}.\\
&\le& c'(s,\tt)\;r^{2 a}\big\|\abs{x}^{s}f\big\|^{\frac{2a}{s}}_{L^2(\Omega,\mu)}+ r^{-\frac{2a\beta}{s}}\big\||\xi|^{\beta}\tt(f)\big\|^{\frac{2a}{s}}_{L^2(\widehat{\Omega},\widehat{\mu})}.
\end{eqnarray*}
The desired result follows by minimizing the
right hand side of that inequality over $r > 0$.
\end{proof}

Inequality \eqref{eqhe.} has been obtained by Cowling and Price \cite{cp} for the Fourier transform 
on $\R^d$ and later  generalized in \cite{Mar} for any pair of positive self-adjoint operators on a Hilbert space.
In particular when $s=\beta=1$ we obtain a version of Heisenberg's uncertainty principle for the operator $\tt$.
Moreover if the function $f\in L^2(\Omega,\mu)$ is supported in a subset $S$ of finite measure one can easily
obtain bounds on $\tt(f)$
that limit the concentration of $\tt(f)$  in any small set and may provide lower bounds for the concentration of
$\tt(f)$ in sufficiently large sets.
For instance we have this simple local uncertainty inequality : if $f$ is supported in a set $S$ with finite measure
$\mu_{2m}(S)<\infty$, then
\begin{eqnarray}
\|\tt(f)\|^2_{L^2(\Sigma ,\widehat{\mu})}
&\le& \widehat{\mu}_{2\widehat{m}}(\Sigma) \|\tt(f)\|^2_{\infty,\widehat{m}}
\le c_{\tt}^2\widehat{\mu}_{2\widehat{m}}(\Sigma)\|f\|^2_{L^1(\Omega,\mu_m)}\nonumber\\
&\le& c_{\tt}^2\mu_{2m}(S)\widehat{\mu}_{2\widehat{m}}(\Sigma)\|f\|^2_{L^2(\Omega,\mu)},\label{eqlocst}
\end{eqnarray}
which implies that the pair  $(S,\Sigma)$ is strongly annihilating provided that $
\mu_{2m}(S)\widehat{\mu}_{2\widehat{m}}(\Sigma)<c_{\tt}^{-2}$.
In the next section we will prove this result for arbitrary subsets $S$ and $\Sigma$ of finite measure.

\section{Pairs of sets of finite measure are strongly annihilating}

In this section we will show that, if $S\subset\Omega$, $\Sigma\subset\widehat{\Omega}$ are sets of finite measure $ 0<\mu_{2m}(S),\;\widehat{\mu}_{2\widehat{m}}(\Sigma)<\infty$, then the pair $(S,\Sigma)$
is strongly annihilating for the operator $\tt$. 
In order to prove this,
we will need to
introduce a pair of orthogonal projections on $L^2(\Omega,\mu)$ defined by
\begin{equation*}
E_S f= \chi_S f,\hspace{1cm} F_\Sigma= \tt^{-1}E_\Sigma \tt,
\end{equation*}
where  $S\subset\Omega$ and $\Sigma\subset\widehat{\Omega}$ are measurable subsets.

We will need the following well-known lemma  ({\it see e.g.} \cite[Lemma 4.1]{GJ}):

 \begin{lemma}\ \\ \label{lemma1}
If $\|E_SF_\Sigma\|:=\|E_SF_\Sigma\|_{L^2(\Omega,\mu)\to L^2(\Omega,\mu)}<1$, then
\begin{equation}\label{stfin}
\|f\|_{L^2(\Omega,\mu)}^2\le \left(1-\|E_SF_\Sigma\|\right)^{-2}\left( \|E_{S^c}f\|_{L^2(\Omega,\mu)}^2
+\|F_{\Sigma^c} f\|_{L^2(\Omega,\mu)}^2 \right).
\end{equation}
\end{lemma}

Unfortunately, showing that $\|E_{S}F_\Sigma\|<1$ is in general difficult. However, the Hilbert-Schmidt norm
$\|E_{S}F_\Sigma\|_{HS}$ is much easier to compute.
Let us illustrate this fact by showing that,
if $S$ and $\Sigma$ are subsets with sufficiently small measure then the pair $(S,\Sigma)$ is strongly annihilating.
We can deduce this result easily from \eqref{eqlocst}, but we will give here another proof that we will use later.

\begin{lemma}\ \\
If $\mu_{2m}(S)\widehat{\mu}_{2\widehat{m}}(\Sigma)<c_{\tt}^{-2}$, then for all function $f\in L^2(\Omega,\mu)$,
\begin{equation*}
\|f\|^2_{L^2(\Omega,\mu)}\le \left(1-c_{\tt}\sqrt{\mu_{2m}(S)\widehat{\mu}_{2\widehat{m}}(\Sigma)}\right)^{-2}
\Big(\|f\|^2_{L^2(S^c,\mu)} +\|\tt(f)\|^2_{L^2(\Sigma^c,\widehat{\mu})}  \Big).
\end{equation*}
\end{lemma}

\begin{proof}
We have, for 
 $f\in L^2(\Omega,\mu)$, $\abs{\tt(f)(\eta)}\leq c_\tt(1+|\xi|)^{\widehat{m}}\norm{f}_{\infty,m}$
thus if $\widehat{\mu}_{2\widehat{m}}(\Sigma)<\infty$, $\chi_\Sigma(\eta) \tt(f)(\eta)\in
L^1(\widehat{\Omega},\widehat{\mu}_{\widehat{m}})$. The Inversion Formula for $\tt$
thus gives
\begin{eqnarray*}
E_SF_\Sigma f(y)&=& \chi_S(y)\int_{\widehat{\Omega}} \chi_\Sigma(\eta) \tt(f)(\eta)\overline{\kk(y,\eta)}
\d \widehat{\mu}(\eta) \\
 &=& \chi_S(y)\int_{\widehat{\Omega}} \chi_\Sigma(\eta) \left(\int_{\Omega} f(x) \kk(x,\eta)\d
 \mu(x)\right)\overline{\kk(y,\eta)}\d \widehat{\mu}(\eta) \\
&=& \int_{\Omega} f(x)\nn(x,y)\d \mu(x),
\end{eqnarray*}
where
\begin{eqnarray*}
\nn(x,y)&=&\chi_S(y) \int_{\widehat{\Omega}} \chi_\Sigma(\eta)\kk(x,\eta)\overline{\kk(y,\eta)}\d\widehat{\mu}(\eta) \\
&=& \chi_S(y) \overline{\int_{\widehat{\Omega}} \chi_\Sigma(\eta)\kk(y,\eta)\overline{\kk(x,\eta)} \d \widehat{\mu}(\eta)}\\
&=&\chi_S(y) \overline{\tt^{-1}\left[\chi_\Sigma(\cdot) \kk(y,\cdot)\right](x)}.
\end{eqnarray*}
Here we appealed repeatedly to Fubini's theorem which is justified by the fact that
$\widehat{\mu}_{2\widehat{m}}(\Sigma)<\infty$ and $\kk$ is bounded by $c_\tt(1+|x|)^m(1+|\xi|)^{\widehat{m}}$.

This shows that $E_SF_\Sigma$ is an integral operator with kernel $\nn$. But, with Plancherel's theorem,
\begin{eqnarray*}
\norm{\nn}^2_{L^2(\Omega,\mu)\otimes L^2(\Omega,\mu)}
&=&\int_{\Omega} |\chi_S(y)|^2\left(\int_{\Omega }\abs{\tt^{-1}\left[\chi_\Sigma(\cdot) \kk(y,\cdot)\right](x)}^2
\d \mu(x)\right) \d \mu(y)\\
&=&\int_{\Omega} |\chi_S(y)|\left(\int_{\widehat{\Omega} }\abs{\chi_\Sigma(\eta) \kk(y,\eta)}^2
\d\widehat{\mu}(\eta)\right) \d \mu(y)\\
&\leq&c_{\tt}^2\mu_{2m}(S)\widehat{\mu}_{2\widehat{m}}(\Sigma)
\end{eqnarray*}
since $|\kk(y,\eta)|\le c_{\tt}(1+|y|)^m(1+|\eta|)^{\widehat{m}}$.
It follows that the Hilbert-Schmidt norm of $E_SF_\Sigma$ is bounded:
\begin{equation}
\label{finite}
\|E_SF_\Sigma\|_{HS}=\norm{\nn}_{L^2(\Omega,\mu)\otimes L^2(\Omega,\mu)}
\le c_{\tt}\sqrt{\mu_{2m}(S)\widehat{\mu}_{2\widehat{m}}(\Sigma)}.
\end{equation}
Now using the fact that $\|E_SF_\Sigma\|\le\|E_SF_\Sigma\|_{HS}$,
we obtain
$$
\|E_SF_\Sigma\|\le c_{\tt}\sqrt{\mu_{2m}(S)\widehat{\mu}_{2\widehat{m}}(\Sigma)}.
$$
It follows from Lemma \ref{lemma1} that
\begin{equation}
\label{eqst}
    \|f\|_{L^2(\Omega,\mu)}^2 \le \left(1-c_{\tt}\sqrt{\mu_{2m}(S)\widehat{\mu}_{2\widehat{m}}(\Sigma)}\right)^{-2}
\Big(\|E_{S^c}f\|_{L^2(\Omega,\mu)}^2+\|F_{\Sigma^c}f\|_{L^2(\Omega,\mu)}^2\Big).
\end{equation}
Plancherel's theorem then gives $\|F_{\Sigma^c}f\|^2_{L^2(\Omega,\mu)}=\|\tt(f)\|^2_{L^2(\Sigma^c,\widehat{\mu})}$ which allows to conclude.
\end{proof}

\begin{remarknum}\ \\
Let $S$, $\Sigma$ be two sets with $\mu_{2m}(S),\widehat{\mu}_{2\widehat{m}}(\Sigma)< \infty$.
Let $\eps_1,\eps_2>0$. Assume that there is a function
$f\in L^2(\Omega,\mu)$ with $\|f\|_{L^2(\Omega,\mu)}=1$ that is $\eps_1$-concentrated on $S$,
\ie $\|E_{S^c}f\|_{L^2(\Omega,\mu)}\le \eps_1$
and $\eps_2$-bandlimited
on $\Sigma$ for the transformation $\tt$, \ie
$\|F_{\Sigma^c}f\|_{L^2(\Omega,\mu)}\le \eps_2$.

Then either $\mu_{2m}(S)\widehat{\mu}_{2\widehat{m}}(\Sigma)\geq c_\tt^{-2}$ or we may
apply Inequality \eqref{eqst} and obtain
$$
1- c_{\tt}\sqrt{\mu_{2m}(S)\widehat{\mu}_{2\widehat{m}}(\Sigma)}\le \sqrt{\eps_1^2+\eps_2^2}.
$$
In both cases, we obtain
\begin{equation}\label{eqdseps}
\mu_{2m}(S)\widehat{\mu}_{2\widehat{m}}(\Sigma)\ge c_{\tt}^{-2}\left(1-\sqrt{\eps_1^2+\eps_2^2}\right)^2,
\end{equation}
which is Donoho-Stark's uncertainty inequality for the integral operator $\tt$.
This inequality improves slightly the result of de Jeu \cite{dj}. In the case of the Fourier transform, it dates back
to Donoho and Stark \cite{DS} in a slightly weaker form and to \cite{hl} to the form \eqref{eqdseps}.
 \end{remarknum}

\medskip
%
%
%
%
Before proving our main theorem, we will now prove the following lemma which results directly from a similar
result in \cite{GJ} for functions in $\cc_0(\R^+)$.

\begin{lemma}\label{libre}\ \\
Let $f$ be a function in $L^2(\Omega,\mu)$ and assume that $0<\mu(\supp \,f)< \infty$.
Then the dilates 
$\{\dd_\lambda f\}_{\lambda>0}$ are linearly independent.
\end{lemma}

\begin{proof}
Let $\zeta\in\S^{d-1}\cap\Omega$ and consider
$$f_\zeta(t)=\dst\begin{cases}t^{a-1/2}f(t\zeta),&\mbox{for}\;t>0;\\ 0,&\mbox{for}\;t<0.\end{cases}$$
For $\zeta\in\S^{d-1}\cap\Omega^c$, we just define $f_\zeta=0$.

Then, there exists $\zeta$ such that $f_\zeta\in L^2(\R)$ and $0<|\mbox{supp}\,f_\zeta|< \infty$,
in particular, $f_\zeta\in L^1(\R)$.
Indeed the first property holds for almost every $\zeta$ since
$$
\int_{\S^{d-1}}\int_\R|f_\zeta(t)|^2\,\mbox{d}t\,Q(\theta)\,\mbox{d}\sigma(\theta)
=\norm{f}^2_{L^2(\Omega,\mu)}< \infty.
$$
As for the second one, notice that
$$
|\supp f_\zeta|\leq |[0,1]|+\int_{\supp f_\zeta\cap[1, \infty)}r^{2a-1}\,\mbox{d}r.
$$
Integrating with respect to $Q(\zeta)\,\mbox{d}\sigma(\zeta)$ we get
$$
\int_{\S^{d-1}\cap\Omega}|\supp f_\zeta|Q(\zeta)\,\mbox{d}\sigma(\zeta)\leq
\norm{Q}_{L^1(\S^{d-1}\cap\Omega)}+\mu\left(\supp f\cap \{|x|>1\}\right)< \infty.
$$
We thus proved that $|\mbox{supp}\,f_\zeta|< \infty$ for almost every $\zeta$. Finally,
$|\mbox{supp}\,f_\zeta|>0$ on a set of $\zeta$'s of positive $\mbox{d}\sigma$ measure, otherwise the support of $f$
would have Lebesgue measure $0$, thus $\mu$-measure zero.

\medskip

Now assume that we had a vanishing linear combination of dilates of $f$:
\begin{equation}
\label{eq:linindphil}
\dst \sum_{finite} \alpha_i f(x/\lambda_i)=0.
\end{equation}
Then, for $t>0$ and the above $\zeta$
$$
\sum_{finite} \alpha_i  \left(\frac{\lambda_i}{t}\right)^{a-1/2}\left(\frac{t}{\lambda_i}\right)^{a-1/2}f\left(\frac{t}{\lambda_i}\zeta\right)=
\frac{1}{t^{a-1/2}}\sum_{finite}\beta_i f_\zeta(t/\lambda_i)=0
$$
where we have set $\beta_i=\alpha_i\lambda_i^{a-1/2}$. Thus 
 $$\sum_{finite}\beta_i f_\zeta(t/\lambda_i)=0.$$
Taking the Euclidean Fourier transform $\ff$, we obtain
$$
\dst \sum_{finite} \beta_i\lambda_i \ff(f_\zeta)(\lambda_i x)=0.
$$
But, as  
$f_\zeta\in L^1(\R)$, it follows from Riemann-Lebesgue's Lemma that 
$\ff(f_\zeta)\in\cc_0(\R)$. It remains to invoke \cite[Lemma 2.1]{GJ} to see that the dilates of 
$\ff(f_\zeta)$ are linearly independent so that
the $\beta_i$'s thus the $\alpha_i$'s are $0$.
\end{proof}

We can now state our main theorem:

\begin{theorem}\ \\
\label{weak}
Let $S\subset\Omega$, $\Sigma\subset\widehat{\Omega}$ be a pair of measurable subsets
with $0<\mu_{2m}(S),\widehat{\mu}_{2\widehat{m}}(\Sigma)<\infty$. Then
any  function $f\in L^2(\Omega,\mu)$ vanishes as soon as $f$ is supported in $S$ and $\tt(f)$ is supported in $\Sigma$.
In other words, $(S,\Sigma)$ is a weak annihilating pair.
\end{theorem}

\begin{proof}
We will write $E_S \cap F_\Sigma$ for the orthogonal projection onto the intersection of the ranges of $E_S$
 and $ F_\Sigma$ and we denote by $\Im \pp$ the range of a linear operator $\pp$.

First we will need the following elementary fact on Hilbert-Schmidt operators:
\begin{equation}\label{hs}
\dim (\Im E_S \cap \Im F_\Sigma)=\norm{E_S \cap F_\Sigma}_{HS}^2  \leq \norm{ E_S F_\Sigma}_{HS}^2.
\end{equation}
Since $\mu_{2m}(S),\;\mu_{2\widehat{m}}(\Sigma)< \infty$,  from Inequality \eqref{finite} we deduce that
\begin{equation}\label{dimfini}
\dim (\Im E_S \cap \Im F_\Sigma) <\infty.
\end{equation}

Assume now that there exists $f_0 \neq0$ such that
$S_0:= \supp f_0$ and $\Sigma_0:= \supp \tt(f_0)$
have both finite measure $0<\mu_{2m}(S_0),\;\widehat{\mu}_{2\widehat{m}}(\Sigma_0)<\infty$,
thus also $\mu(S_0)< \infty$ so that Lemma \ref{libre} applies.

Next, let $S_1$ (resp. $\Sigma_1$) be a measurable subset of $\Omega$ (resp. $\widehat{\Omega}$) of finite  measure $0<\mu_{2m}(S_1)<\infty$ (resp. $0<\widehat{\mu}_{2\widehat{m}}(\Sigma_1)<\infty$),
such that $S_0\subset S_1$ (resp. $\Sigma_0\subset\Sigma_1$). Since for $\lambda>0$,
$$
\dst\mu_{2m}(S_1\cup \lambda S_0)=\norm { \chi_{\lambda S_0}-\chi_{S_1} }_{L^2(\Omega,\mu_{2m})}^2+
 \scal{ \chi_{\lambda S_0},\chi_{S_1}}_{L^2(\Omega,\mu_{2m})},
$$
 the function $\lambda \mapsto \mu_{2m}(S_1\cup \lambda S_0)$ is continuous on $ \R^+\backslash \{0\}$.
The same holds for $\lambda \mapsto \widehat{\mu}_{2\widehat{m}}(\Sigma_1\cup \lambda \Sigma_0)$.

From  this, one easily deduces that, there exists an infinite sequence of distinct numbers
 $\dst(\lambda_j)_{j=0}^{\infty} \subset \R^+\backslash \{0\}$
with $\lambda_0=1$, such that, if we denote by $S=\dst\bigcup_{j=0}^{\infty} \lambda_jS_0$ and
$\Sigma=\dst\bigcup_{j=0}^{\infty}  \frac{1}{\lambda_j}\Sigma_0$,
\begin{equation*}\label{sequence}
 \mu_{2m}(S)< 2 \mu(S_0), \quad \widehat{\mu}_{2\widehat{m}}(\Sigma)< 2\widehat{\mu}_{2\widehat{m}}(\Sigma_0).
\end{equation*}
We next define $f_i= \dd_{\lambda_i}f_0$, so that $\supp f_i= \lambda_i S_0\subset S$.
Since $ \tt(f_i)=  \lambda_i^{a-\widehat{a}}\widehat{\dd}_{\frac{1}{\lambda_i}}\tt(f_0)$, we have
$\supp \tt(f_i)= \frac{1}{\lambda_i} \Sigma_0\subset\Sigma$.

As 
$\supp f_0$ has finite measure, it follows from 
Lemma \ref{libre} that $(f_i)_{i=0}^\infty$ are linearly independent vectors
belonging to $(\Im E_{S} \cap \Im F_{\Sigma})$, which contradicts \eqref{dimfini}.
\end{proof}

\begin{remarknum}\ \\
{\rm The theorem can be extended to operators $\tt$ that take their values in a finite dimensional
Banach algebra.

The proof given here follows roughly the scheme of Amrein-Berthier's original one in \cite{AB}.
It can obviously be adapted so as to replace dilations by actions of more general groups on measure spaces.
The main difficulty would be to prove that this action leads to linearly independent functions as in Lemma \ref{libre}.
As we have no specific application in mind, we refrain from stating a more general result.}
\end{remarknum}

A simple well known functional analysis argument allows us to obtain the following improvement ({\it see e.g.} \cite[Proposition 2.6]{BD}):

\begin{corollary}\label{str}\ \\
Let $S\subset\Omega$, $\Sigma\subset\widehat{\Omega}$ be a pair of measurable subsets of  finite measure,
$0<\mu_{2m}(S),\, \widehat{\mu}_{2\widehat{m}}(\Sigma)<\infty$. Then there exists a constant $C(S,\Sigma)$ such
that for all $f\in L^2(\Omega,\mu)$,
\begin{equation*}
\|f\|^2_{L^2(\Omega,\mu)}\le C(S,\Sigma)\Big(\|f\|^2_{L^2(S^c,\mu)}
+\|\tt(f)\|^2_{L^2(\Sigma^c,\widehat{\mu})}  \Big).
\end{equation*}
\end{corollary}

\begin{proof}
Assume there is no such constant $D(S,\Sigma)$ such that for every function $f\in L^2(\Omega,\mu)$ supported in $S$,
$$
 \|f\|^2_{L^2(\Omega,\mu)}\le D(S,\Sigma)\|\tt(f)\|^2_{L^2(\Sigma^c,\widehat{\mu})}.
$$
Then, there exists a sequence $f_n \in L^2(\Omega,\mu)$ with $\norm{f_n}_{L^2(\Omega,\mu)}=1$
and with support in $S$ such that $\norm{E_{\Sigma^c}\tt(f_n) }_{L^2(\widehat{\Omega},\widehat{\mu})}$ converge to $0$.
Moreover, we may assume that $f_n$ is weakly convergent in
$L^2(\Omega,\mu)$ with some limit $f$. As $\tt(f_n)(\xi)$ is the scalar product of $f_n$
and $\overline{E_S\kk(\cdot,\xi)}$, it follows that $\tt(f_n)$ converge to $\tt(f)$. Finally, as
$|\tt(f_n)(\xi)|^2$ is bounded by $c_{\tt}^2\mu_{2m}(S)(1+|\xi|)^{2\widehat{m}}$, we
may apply Lebesgue's theorem, thus $E_\Sigma \tt(f_n)$ converges to $E_\Sigma\tt(f)$
in $L^2(\widehat{\Omega},\widehat{\mu})$.
But we have $\supp f \subset S$ and $\supp \tt(f)\subset\Sigma$ so by Theorem \ref{weak},
$f$ is $0$, which contradicts the fact that $f$ has norm $1$.
\end{proof}

Now we will show a global uncertainty inequality type for the transformation $\tt$. But this time we will use Corollary \ref{str} and the proof here
 is simpler than that using the local uncertainty principle and not necessary with the same constant.

\begin{corollary}\ \\
Let $s,\;\beta>0$. Then there exists a constant 
$C=C(s,\beta,a)$ such that for all $f\in L^2(\Omega,\mu)$,
\begin{equation}\label{eqhe}
\big\||x|^{s}f\big\|^{\frac{2\beta}{s+\beta}}_{L^2(\Omega,\mu)} \;
\big\||\xi|^{\beta}\tt(f)\big\|^{\frac{2s}{s+\beta}}_{L^2(\widehat{\Omega},\widehat{\mu})}
\ge C \norm{f}^{2}_{L^2(\Omega,\mu)}.
\end{equation}
\end{corollary}

\begin{proof}
Let $B_1=\Omega\cap\{x\,:|x|\leq 1\}$ and $\widehat{B}_1=\widehat{\Omega}\cap\{\xi\,:|\xi|\leq 1\}$.
Let $B_1^c=\Omega\setminus B_1$ and $\widehat{B}^c_1=\widehat{\Omega}\setminus\widehat{B}_1$.

From Corollary \ref{str} there exists a constant $C=C(B_1,\widehat{B_1})$ such that
$$
\|f\|_{L^2(\Omega,\mu)}^2\le C\left(\|f\|_{L^2(B_1^c,\mu)}^2 +\|\tt(f)\|_{L^2(\widehat{B}^c_1,\widehat{\mu})}^2\right).
$$
It follows then
\begin{eqnarray*}
\|f\|_{L^2(\Omega,\mu)}^2 &\le&  C\left(\big\||x|^sf\big\|_{L^2(B_1^c,\mu)}^2
+\big\||\xi|^\beta\tt(f)\big\|_{L^2(\widehat{B}_1^c,\widehat{\mu})}^2 \right)\\
&\le& C \left(\big\||x|^sf\big\|_{L^2(\Omega,\mu)}^2
+\big\||\xi|^\beta\tt(f)\big\|_{L^2(\widehat{\Omega},\widehat{\mu})}^2 \right).
\end{eqnarray*}
Replacing $f$ by $\dd_\lambda f$ in the last inequality we have by \eqref{dt}
\begin{equation}\label{eqd}
\big\|\dd_\lambda f\big\|_{L^2(\Omega,\mu)}^2\le C \left(\big\||x|^s \dd_\lambda f\big\|_{L^2(\Omega,\mu)}^2
+
\big\||\xi|^\beta \widehat{\dd}_{\frac{1}{\lambda}}\tt(f)\big\|_{L^2(\widehat{\Omega},\widehat{\mu})}^2 \right),
\end{equation}
which gives
$$
\big\|f\big\|_{L^2(\Omega,\mu)}^2\le C \left(\lambda^{2s}\big\||x|^s f\big\|_{L^2(\Omega,\mu)}^2
+ \lambda^{-2\beta} \big\||\xi|^\beta \tt(f)\big\|_{L^2(\widehat{\Omega},\widehat{\mu})}^2 \right).
$$
The desired result follows by minimizing the right hand side of that inequality over $\lambda > 0$.
\end{proof}

Let us notice that Theorem \ref{weak} is valid in the $L^1$-version. Precisely we have the following proposition:

\begin{proposition}\ \\
\label{weak2}
Let $S\subset\Omega$, $\Sigma\subset\widehat{\Omega}$ be a pair of measurable subsets with $0<\mu_{2m}(S),\widehat{\mu}_{2\widehat{m}}(\Sigma)<\infty$.
Suppose $f\in L^1(\Omega,\mu_m)$ 
(in particular $f\in L^1(\Omega,\mu)$) verifies $\supp f\subset S$ and $\supp\tt(f)\subset\Sigma$, then $f=0$.
\end{proposition}

\begin{proof}
If $f\in L^1(\Omega,\mu_m)$, then $(1+|\xi|)^{-\widehat{m}}\tt (f) \in L^\infty(\widehat{\Omega},\widehat{\mu})$.
Then
$$
\norm{\tt(f)}_{L^1(\widehat{\Omega},\widehat{\mu}_{\widehat{m}})}
=\norm{ \chi_\Sigma \tt(f)}_{L^1(\widehat{\Omega},\widehat{\mu}_{\widehat{m}})}
\le \widehat{\mu}_{2\widehat{m}}(\Sigma)\norm{(1+|\xi|)^{-\widehat{m}}\tt(f)}_{L^\infty(\widehat{\Omega},\widehat{\mu})}<\infty.
$$
This implies that $\tt(f) \in L^1(\widehat{\Omega},\widehat{\mu}_{\widehat{m}})$,
thus $(1+|x|)^{-m}f\in L^\infty(\Omega,\mu)$.
Finally,
\begin{eqnarray*}
\norm{f}_{L^2(\Omega,\mu)}^2&=&\int_\Omega (1+|x|)^{-m}|f(x)||f(x)|(1+|x|)^m\,\mbox{d}\mu(x)\\
&\leq& \norm{(1+|x|)^{-m}f}_{L^\infty(\Omega,\mu)}\norm{f}_{L^1(\Omega,\mu_m)}<\infty,
\end{eqnarray*}
hence $f \in L^2(\Omega,\mu)$. By  Theorem \ref{weak}
we have $f=0$.
\end{proof}

The same argument as the one used in the proof of Corollary \ref{str} gives the following result\,:

\begin{proposition}\ \\
Let $S\subset\Omega$, $\Sigma\subset\widehat{\Omega}$ be a pair of measurable subsets with $0<\mu_{2m}(S),\widehat{\mu}_{2\widehat{m}}(\Sigma)<\infty$.
Then there exists a constant
$D(S,\Sigma)$ such that for all function $f\in L^1(\Omega,\mu)$ supported in $S$,
\begin{equation*}
 \|f\|_{L^1(\Omega,\mu)}\le D(S,\Sigma) \|\tt(f)\|_{L^1(\Sigma^c,\widehat{\mu})}.
\end{equation*}
\end{proposition}

\section{Examples}
\subsection{The Fourier transform and the Fourier-Bessel transform}\

Let $\d\mu(x)=(2\pi)^{-d/2}\d x$ the Lebesgue measure and $\tt=\ff$ the Fourier transform. For $f\in L^1(\R^d,\mu)\cap L^2(\R^d,\mu)$, the Fourier transform is defined by
$$
\ff(f)(\xi)=\int_{\R^d}f(x) e^{-i\scal{x,\xi}} \d\mu(x), \hspace{0,5cm} \xi\in \R^d;
$$
and is then extended to all $L^2(\R^d,\mu)$ in the usual way. In this case we take $c_{\tt}=1$, $a=d/2$  and
$m=\widehat{m}=0$. Then \eqref{eqdseps}
is Donoho-Stark's theorem \cite{DS, hl}, Corollary \ref{str} is Amrein-Berthier's theorem \cite{AB} while 
the local and the global uncertainty principles for the Fourier transform date back respectively to
\cite{Pr, Prr} and \cite{cp}.
Note that our proof of Theorem \ref{weak} is inspired by the one established in \cite{AB} where we replace
translation by dilation.

If $f(x)=f_0(|x|)$ is a radial function on $\R^d$, then
$$
\ff(f)(\xi)=\frac{1}{2^{d/2-1}\Gamma(d)}\int_0^\infty f_0(t)j_{d/2-1}(t|\xi|)t^{d-1}\d t=\ff_{d/2-1}(f_0)(|\xi|),
$$
where $\ff_{d/2-1}$ is the Fourier-Bessel transform of index $d/2-1$. For $\alpha\ge-1/2$, $j_\alpha$ is  the Bessel function given by
$$
j_\alpha(x)=2^\alpha\Gamma(\alpha+1)
\frac{J_\alpha (x)}{x^\alpha}:=\Gamma(\alpha+1)\sum_{n=0}^{\infty}\frac{(-1)^n}{n!\Gamma(n+\alpha+1)}\left(\frac{x}{2}\right)^{2n},
$$
where $\Gamma$ is the gamma function.

We have $|j_\alpha|\le 1$ and if we denote $\d \mu_\alpha (x)=\frac{1}{2^\alpha\Gamma(\alpha+1)}x^{2\alpha+1}\,\d x$, then
for $f\in L^1(\R^+, \mu_\alpha)\cap L^2(\R^+,\mu_\alpha)$, the Fourier-Bessel (or Hankel) transform is defined by
$$
\ff_\alpha(f) (\xi)=\int_0^\infty f(x)j_\alpha(x\xi)\d\mu_\alpha (x),\ \ \xi\in \R^+;
$$
and extends to an isometric isomorphism on $ L^2(\R^+,\mu_\alpha)$ with $\ff_\alpha^{-1}=\ff_\alpha$. Theorem $A$ and  Theorem $B$ has been stated  in \cite{GJ} for this transformation.
Moreover we have the following two new results.

\begin{theorem}[Donoho-Stark's uncertainty principle for $\ff_\alpha$]\ \\
Let $S$, $\Sigma$  be a pair of measurable subsets of $\R^+$ and $\alpha>-1/2$.
If $f\in L^2(\R^+,\mu_\alpha)$ of unit $L^2$-norm is $\eps_1$-concentrated on $S$ and $\eps_2$-bandlimited on
$\Sigma$ for the the Fourier-Bessel transform, then
\begin{equation}
\mu_\alpha(S)\mu_\alpha(\Sigma)\ge\left(1-\sqrt{\eps_1^2+\eps_2^2}\right)^2
\quad\mbox{and}\quad
|S||\Sigma|\ge c_\alpha\left(1-\sqrt{\eps_1^2+\eps_2^2}\right)^2,
\end{equation}
where $c_\alpha$ is a numerical constant that depends only on $\alpha$.
\end{theorem}

This result improves the estimate in \cite{karoui} (which has already improved \cite{tuan})
showing that, if $f$ of unit $L^2$-norm is $\eps_1$-concentrated on $S$ and $\eps_2$-bandlimited on $\Sigma$, then
\begin{equation*}
 |S||\Sigma|\ge c'_\alpha(1-\eps_1-\eps_2)^2.
\end{equation*}

\begin{theorem}[Global uncertainty principle for $\ff_\alpha$]\ \\
For $s,\;\beta>0$, there exists a constant $C_{s,\beta,\alpha}$ such that for all $f\in L^2(\R^+,\mu_\alpha)$,
  \begin{equation*}
    \big\|x^{s}f\big\|^{\frac{2\beta}{s+\beta}}_{L^2(\R^+,\mu_\alpha)} \; \big\|\xi^{\beta}\ff_\alpha(f)\big\|^{\frac{2s}{s+\beta}}_{L^2(\R^+,\mu_\alpha)}
    \ge C_{s,\beta,\alpha} \norm{f}^{2}_{L^2(\R^+,\mu_\alpha)}.
  \end{equation*}
\end{theorem}
The case when $s=\beta=1$ has been established in \cite{Bo, rousler} with the optimal constant $C_{1,1,\alpha}=\alpha+1$.

\subsection{The Fourier-Dunkl transform}\

In this section we will deduce new uncertainty principles for the Dunkl transform.
Uncertainty principles for this  transformation have been considered in various places, {\it e.g.} \cite{Rm, sim} for a Heisenberg type inequality or
\cite{GT} for Hardy type uncertainty principles  and recently \cite{cdkm, maj} for a generalization and a variant of Cowling-Price's theorem, Beurling's theorem,
 Miyachi's theorem and Donoho-Stark's uncertainty principle.

Let us fix some notation and present some necessary material on the Dunkl transform.
Let $G$ be a finite reflection group on $\R^d$, associated with a root system $R$ and $R_+$ the positive
subsystem of $R$ ({\it see} \cite{dje, Du, RV}). We denote by $k$ a nonnegative multiplicity function defined on $R$
with the property that $k$ is $G$-invariant. We associate with $k$ the index
$$
\gamma:=\gamma(k)= \sum _{\xi \in R_+} k(\xi)\ge0
$$
and the weight function $w_k$ defined by
$$
w_k(x)=\prod _{\xi \in R_+} |\scal{\xi ,x}|^{2 k(\xi)}.
$$
Further we introduce the Mehta-type constant $c_k$ by
$$
c_k= \left(\int_{\R^d} e^{-\frac{\abs{x}^2}{2}}\d \mu_k(x)\right)^{-1},
$$
where\footnote{we chose here to stick to the notation that is usual in Dunkl analysis rather than that of the previous section in which $\mu_k$ is simply denoted by $\mu$.} $\d \mu_k(x)= w_k(x) \d x$. Moreover
$$
\int_{\S^{d-1}} w_k(x) \d \sigma (x)= \dst\frac{c_k^{-1}}{2^{\gamma+d/2-1}\Gamma(\gamma+d/2)}=d_k.
$$

By using the homogeneity of $w_k$ it is shown in \cite{RV} that for a radial function $f\in L^1(\R^d,\mu_k)$
the function $\widetilde{f}$ defined on $\R^+$ by $f(x) =\widetilde{f}(\abs{x})$, for all $x \in \R^d$ is integrable with respect to the measure
 $r^{2\gamma+d-1}\d r$. More precisely,
\begin{eqnarray} \label{radial}
 \int_{\R^d} f(x) w_k(x) \d x &=&  \int_{\R^+} \left(\int_{\S^{d-1}} w_k(ry) \d \sigma (y)\right) \widetilde{f}(r)r^{d-1}\d r \nonumber\\
&=&  \dst d_k\int_{\R^+} \widetilde{f}(r)r^{2\gamma+ d-1}\d r.
\end{eqnarray}

Introduced by C.\,F. Dunkl in \cite{D}, the Dunkl operators $T_j$, $1 \leq j \leq d$ on $\R^d$ associated with
the reflection group $G$ and the multiplicity function $k$ are the first-order differential-difference
operators given by
$$
T_j f(x)= \frac{\partial f}{\partial x_j}+ \sum_{\xi \in R_+} k(\xi) \xi_j \frac{f(x)- f(\sigma_{\xi} (x))}{\scal{\xi,x}},
\hspace{0,5cm} x \in \R^d;
$$
where $f$ is an infinitely differentiable function on $\R^d$, $\xi_j= \scal{\xi,e_j}$, $(e_1, \ldots, e_d)$ being the canonical basis of $\R^d$ and  $\sigma_{\xi}$
 denotes the reflection with respect to the hyperplane orthogonal to $ \xi$.

The Dunkl kernel $\kk_k$ on $\R^d \times \R^d$ has been introduced by C.\,F. Dunkl in \cite{Du}. For $y \in \R^d$ the
function $x \mapsto\kk_k(x, y)$ can be viewed as the solution on $\R^d$ of the following initial problem
$$
T_ju(x,y) = y_j u(x,y); \hspace{0,25cm} 1 \leq j \leq d, \hspace{0,5cm} u(0,y)=1.
$$
This kernel has a unique holomorphic extension to $\C^d \times \C^d$. M. R\"{o}sler has proved in \cite{R} the
following integral representation for the Dunkl kernel
$$
\kk_k(x,z)= \int_{\R^d} e^{\scal{y,z}}\d \mu_x^k (y), \hspace{0,5cm}x \in \R^d, \hspace{0,5cm} z \in \C^d;
$$
where $\mu_x^k$ is a probability measure on $\R^d$ with support in the closed ball $B_{\abs{x}}$.
We have ({\it see} \cite{R}) for all $\lambda \in \C$, $z,\; z'\in \C^d$ and $x, \;y \in \R^d$
$$
\kk_k(z, z')= \kk_k(z', z),\hspace{0,25cm} \kk_k(\lambda z, z') = \kk_k(z, \lambda z'),\hspace{0,25cm}\overline{\kk_k(-iy, x)}=\kk_k(iy, x),\hspace{0,25cm} |\kk_k(-iy, x)|\leq 1.
$$
The Dunkl transform $\ff_k$ of a function $f \in L^1(\R^d,\mu_k)\cap L^2(\R^d,\mu_k)$ which was introduced by C. F. Dunkl ({\it see} \cite{dje, Dun}), is given by
$$
\ff_k (f)(\xi):=c_k \int_{\R^d}\kk_k(-i\xi, x) f(x)\d \mu_k(x),  \hspace{0,5cm} \xi \in \R^d;
$$
and extends uniquely to an isometric isomorphism on $L^2(\R^d,\mu_k)$ with $\ff_k^{-1}(f)(\xi)=\ff_k (f)(-\xi)$.

The Dunkl transform  $\ff_k$ provides a natural generalization of the  Fourier transform $\ff$, to which it reduces in
the case $k =0$, and if $f(x)=\widetilde{f}(|x|)$
is a radial function on $\R^d$, then $$\ff_k (f)(\xi)=\ff_{\gamma+d/2-1}(\widetilde{f})(|\xi|),$$
where $\ff_{\gamma+d/2-1}$ is the Fourier-Bessel transform of index $\gamma+d/2-1$.

Now if we take $c_{\tt}=c_k$, $a=\gamma+ d/2$ and $m=\widehat{m}=0$, then  from Section $2$ and $3$ we obtain a new
uncertainty principles for the Dunkl transform $\ff_k$.

\begin{theorem}[Donoho-Stark's uncertainty principle for $\ff_k$]\ \\
Let $S$, $\Sigma$  be a pair of measurable subsets of $\R^d$. If $f\in L^2(\R^d,\mu_k)$ of unit $L^2$-norm is $\eps_1$-concentrated on $S$ and
$\eps_2$-bandlimited on $\Sigma$ for the Dunkl  transform, then
\begin{equation}\label{eqdsdun}
\mu_k(S)\mu_k(\Sigma)\ge c_k^{-2}\left(1-\sqrt{\eps_1^2+\eps_2^2}\right)^2.
\end{equation}
\end{theorem}

Note that the Donoho-Stark's uncertainty principle has recently been proved in \cite{maj} for the Dunkl transform but our
inequality \eqref{eqdsdun} is a little stronger.

Let us now state how our results translate to the Fourier-Dunkl transform. These results are new to our knowledge.

\begin{theorem}\ \\
Let $S$, $\Sigma$ be a pair of measurable subsets of $\R^d$ with finite measure, $0<\mu_k(S),\, \mu_k(\Sigma)<\infty$.
Then the following uncertainty principles hold.

\begin{enumerate}
\item \emph{Local uncertainty principle for $\ff_k$:}
\begin{enumerate}
  \item For $0 <s < \gamma+ d/2$, there is a constant $c(s,k)$ such that for all $f\in L^2(\R^d,\mu_k)$,
 $$
\norm{\ff_k(f)}_{L^2(\Sigma, \mu_k)} \leq c(s,k) \Big[\mu_k(\Sigma)\Big]^{\frac{s}{2\gamma+d}}\norm{\abs{x}^s f}_{L^2(\R^d,\mu_k)}.
$$
\item For $s > \gamma+ d/2$, there is a constant $c'(s,k)$ such that for all $f\in L^2(\R^d,\mu_k)$,
$$
\norm{\ff_k(f)}_{L^2(\Sigma,, \mu_k)} \leq c'(s,k)\Big[\mu_k(\Sigma)\Big]^{\frac{1}{2}}\norm{f}_{L^2(\R^d,\mu_k)}^{1-\frac{2\gamma+d}{2s}}
\norm{\abs{x}^s f}_{L^2(\R^d,\mu_k)}^{\frac{2\gamma+d}{2s}}.
$$
\end{enumerate}

\item \emph{Benedicks-Amrein-Berthier's uncertainty principle for $\ff_k$:} \\
There exists a constant $C_k(S,\Sigma)$ such that for all $f\in L^2(\R^d,\mu_k)$,
\begin{equation*}
 \|f\|^2_{L^2(\R^d,\mu_k)}\le C_k(S,\Sigma)\Big(\|f\|^2_{L^2(S^c,\mu_k)} +\|\ff_k(f)\|^2_{L^2(\Sigma^c,\mu_k)} \Big).
\end{equation*}

\item \emph{Global uncertainty principle for $\ff_k$:} \\
For $s,\;\beta>0$, there exists a constant $C_{s,\beta,k}$ such that for all $f\in L^2(\R^d,\mu_k)$,
$$
    \big\||x|^{s}f\big\|^{\frac{2\beta}{s+\beta}}_{L^2(\R^d,\mu_k)} \; \big\||\xi|^{\beta}\ff_k(f)\big\|^{\frac{2s}{s+\beta}}_{L^2(\R^d,\mu_k)}
    \ge C_{s,\beta,k} \norm{f}^{2}_{L^2(\R^d,\mu_k)}.
$$
\end{enumerate}
\end{theorem}

A simple computation shows that
$$
c(s,k)=\dst\frac{2\gamma+d}{2\gamma+d-2s}\left[\frac{c_k}{2s} \sqrt{(2\gamma+d-2s)d_k}\right]^{\frac{2\gamma+d}{2s}}
$$
and
$$
c'(s,k)=c_k \left[\frac{d_k}{2\gamma+d}\left(\frac{2s}{2\gamma+d}-1 \right)^{\frac{2\gamma+d}{2s}-1}
\Gamma\left(\frac{2\gamma+d}{2s}\right)\Gamma\left(1-\frac{2\gamma+d}{2s}\right) \right]^{1/2}.
$$

In the particular case $s=\beta=1$ for the global uncertainty principle, we recover
Heinsenberg's inequality for the Dunkl transform but with $C_{1,1,k}\le \gamma+d/2$, where $\gamma+d/2$ is
the optimal constant in the 
Heisenberg uncertainty principle given in \cite{Rm, sim}.

\subsection{The Fourier-Clifford transform}\

Let us now introduce the basics of Clifford analysis that are needed to introduce the Fourier-Clifford transform.
Facts used here can be found {\it e.g.} in \cite{Br1,Br2}. We also follow as closely as possible the presentation of Clifford analysis from \cite{BSS, dbie}.

Throughout this section $d\geq 2$ will be a fixed integer
and the measure $\mbox{d}\mu(x)=\mbox{d}\widehat{\mu}(x)=(2\pi)^{-d/2}\mbox{d}x$ is the Lebesgue measure on $\R^d$.
We first associate the Clifford algebra $Cl_{0,d}(\C)$ generated by the canonical basis $e_j$, $j=1,\ldots,d$.
For $A=\{j_1,j_2,\ldots,j_k\}\subset\{1,\ldots,d\}$ with $j_1<j_2<\cdots<j_k$, we denote by
$e_A=e_{j_1}e_{j_2}\cdots e_{j_k}$.
The basis of the Clifford algebra is then given by $\ee=\bigl\{e_A, A\subset\{1,\ldots,d\}\bigr\}$.
The Clifford algebra is then the complex vector space generated by $\ee$ endowed with the multiplication rule given by
\begin{enumerate}
\renewcommand{\theenumi}{\roman{enumi}}
\item $e_\emptyset=1$ is the unit element

\item $e_j^2=-1$, $j=1,\ldots,d$

\item $e_je_k+e_ke_j=0$, $j,k=1,\ldots,d$, $j\not=k$.
\end{enumerate}

Conjugation is defined by the anti-involution for which $\overline{e_j}=-e_j$, $j=1,\ldots,d$
with the additional rule $\bar i=-i$.

The scalars are then identified with $\span\{e_\emptyset\}$ while we identify a vector $x=(x_1,\ldots,x_d)$
with
$$
\underline{x}=\sum_{j=1}^d e_jx_j.
$$
The product of two vectors splits into a scalar part and a \emph{bivector} part
$$
\ux \uy=-\scal{\ux,\uy}+\ux\wedge\uy
$$
and
$$
\ux\wedge\uy=\sum_{j=1}^d\sum_{k=i+1}^d e_je_k(x_jy_k-x_ky_j).
$$
Note that $\ux^2=-|x|^2$.

The functions defined in this section are defined on $\R^d$ and take their values in the Clifford algebra $Cl_{0,d}(\C)$. We can now introduce the so-called Dirac operator, a first order vector differential operator defined by
$$
\partial_{\ux}=\sum_{j=1}^d\partial_{x_j}e_j.
$$
Its square equals, up to a minus sign, the Laplace operator on $\R^d$, $\partial_{\ux}^2=-\Delta$.
The central notion in Clifford analysis is the notion of \emph{monogenicity}, the higher-dimensional analogue of
holomorphy: a function is called (left)-monogenic if $\partial_{\ux}f=0$.

We will denote by $\mm_k$ the space of all \emph{spherical monogenics} of degree $k$, that is, homogeneous polynomials
of degree $k$ that are null-solutions of the Dirac operator. We fix a basis
$\{M_k^{(\ell)}\}_{\ell=1,2,\ldots,\dim\mm_k}$ of $\mm_k$. Further, the Laguerre polynomials are
denoted by $L_j^\alpha$. We then consider the following functions, called the \emph{Clifford-Hermite} functions
\begin{equation}
\begin{matrix}
\psi_{2j,k,\ell}(\ux)&=&\gamma_{2j,k,\ell}L_j^{\frac{d}{2}+k-1}(|\ux|^2)M_k^{(\ell)}(\ux)e^{-|\ux|^2/2}\\
\psi_{2j+1,k,\ell}(\ux)&=&\gamma_{2j+1,k,\ell}L_j^{\frac{d}{2}+k}(|\ux|^2)\ux M_k^{(\ell)}(\ux)e^{-|\ux|^2/2}
\end{matrix},
\end{equation}
where $j,k\in\N$ and $\ell\in\{1,\ldots,\dim\mm_k\}$.
Provided the $\gamma_{j,k,\ell}$'s are properly chosen, this is an \emph{orthonormal basis} of $L^2(\R^d)$
({\it see} \cite{BdSKS}).

Next, introducing spherical coordinates in $\R^d$: $\ux=r\underline{\omega}$, $r=|\ux|\in\R^+$,
$\underline{\omega}\in\S^{d-1}$, the Dirac operator takes the form
$$
\partial_{\ux}=\underline{\omega}\Bigl(\partial_r+\frac{1}{r}\Gamma_{\ux}\Bigr)
$$
where
$$
\Gamma=\ux\wedge \partial_{\ux}=-\sum_{j=1}^d\sum_{k=j+1}^d e_je_k(x_j\partial_{x_k}-x_k\partial_{x_j})
$$
is the so-called \emph{angular Dirac operator}.

We are now in position to define the Clifford-Fourier transforms on $\ss(\R^d)$. This can be done in three equivalent
ways:

-- $\ff_{\pm}[f]=e^{i d\frac{\pi}{4}}e^{i\frac{\pi}{4}(\Delta-|\ux|^2\mp2\Gamma)}f$;

-- via an integral kernel
$$
\ff_{\pm}[f](\underline{\eta})=\int_{\R^d}f(\ux)K_\pm(\ux,\underline{\eta})\,\mbox{d}\mu(\ux)
$$
where $K_\pm(\ux,\underline{\eta})=\dst
e^{id\frac{\pi}{4}}e^{i\frac{\pi}{2}\Gamma_{\underline{\eta}}}e^{-i\scal{\ux,\underline{\eta}}}$;

-- via its eigenfunctions
$$
\ff_{\pm}[\psi_{2j,k,\ell}]=(-1)^{j+k}(\mp 1)^k\psi_{2j,k,\ell}
\quad\mbox{and}\quad
\ff_{\pm}[\psi_{2j+1,k,\ell}]=i^d(-1)^{j+1}(\mp 1)^{k+d-1}\psi_{2j,k,\ell}.
$$

The third definition immediately shows that $\ff_\pm$ extend to unitary operators on $L^2(\R^d,\mu)$.

The fact that the integral operator definition makes sense on $\ss(\R^d)$ and that the
kernel of the inverse transform is indeed $\overline{K_\pm(\ux,\underline{\eta})}$
has been proved respectively in \cite[Theorem 6.3 and Proposition 3.4]{dbie}.

Finally, the kernel is not known to be polynomially bounded, excepted when the dimension $d$ is even
\cite[Theorem 5.3]{dbie} and then
$$
|K(\ux,\underline{\eta})|\leq C(1+|\ux|)^{(d-2)/2}(1+|\underline{\eta}|)^{(d-2)/2}.
$$
Thus $m=\widehat{m}=(d-2)/2$, $c_\tt=C$ and $a=d/2$.

It remains to notice that all results from the first part of the paper extend with no change to Clifford-valued functions. More precisely,
 we obtain the following results:

\begin{theorem}\ \\
Let $d$ be even and $\d\nu(x)=\dst(1+|x|)^{d-2}\d \mu(x)$.
Let $S$, $\Sigma$ be a pair of measurable subsets of $\R^d$. Then
the Clifford-Fourier transform satisfies the following uncertainty principles.

\begin{enumerate}
\item \emph{Donoho-Stark's uncertainty principle for $\ff_\pm$:}\\
 If $f\in L^2(\R^d, \mu)$ of unit $L^2$-norm is $\eps_1$-concentrated on $S$ and
$\eps_2$-bandlimited on $\Sigma$ for the Clifford-Fourier transform, then
\begin{equation*}
\nu(S)\nu(\Sigma)\ge C^{-2}\left(1-\sqrt{\eps_1^2+\eps_2^2}\right)^2.
\end{equation*}
\item \emph{Local uncertainty principle for $\ff_\pm$:} \\
If $\Sigma$ is subset of finite measure $0< \nu(\Sigma)< \infty$, then
\begin{enumerate}
  \item for $\,0 <s <d/2$, there is a constant $c(s)$ such that for all $f\in L^2(\R^d,\mu)$,
 \begin{equation*}
\norm{\ff_\pm(f)}_{L^2(\Sigma,\widehat{\mu})} \leq
\begin{cases} c(s) \Big[\nu(\Sigma)\Big]^{\frac{s}{2(d-1)}}\norm{\abs{x}^s f}_{L^2(\R^d, \mu)}
,&\mbox{if }\nu(\Sigma)\leq 1;\\
c(s)\Big[\nu(\Sigma)\Big]^{\frac{s}{d}}\big\|\abs{x}^s f\big\|_{L^2(\R^d, \mu)},&\mbox{if }\nu(\Sigma)> 1;
\end{cases}
\end{equation*}
\item for $\,d/2\leq s\leq d-1$ then, for every $\eps>0$ there is a constant $c(s,\eps)$ such that for all $f\in L^2(\R^d,\mu)$,
 \begin{equation*}
\norm{\ff_\pm(f)}_{L^2(\Sigma,\widehat{\mu})}
\leq
\begin{cases}
c(s, \eps)\Big[\nu(\Sigma)\Big]^{\frac{1}{4(1-1/d)}-\eps}\big\|f\big\|_{L^2(\R^d,\mu)}^{1-\frac{d}{2s}+\eps}
\big\|\abs{x}^s f\big\|_{L^2(\R^d,\mu)}^{\frac{d}{2s}-\eps}, &\mbox{if }\nu(\Sigma)\le 1;\\
c(s, \eps)\Big[\nu(\Sigma)\Big]^{\frac{1}{2}-\eps}\big\|f\big\|_{L^2(\R^d,\mu)}^{1-\frac{d}{2s}+\eps}
\big\|\abs{x}^s f\big\|_{L^2(\R^d,\mu)}^{\frac{d}{2s}-\eps}, &\mbox{if }\nu(\Sigma)> 1;
\end{cases}
\end{equation*}
\item for $  s > d-1$, there is a constant $c'(s)$ such that for all $f\in L^2(\R^d, \mu)$,
$$
\norm{\ff_\pm(f)}_{L^2(\Sigma, \mu)} \leq c'(s)\Big[\nu(\Sigma)\Big]^{\frac{1}{2}}\norm{(1+\abs{x}^s)f}_{L^2(\R^d, \mu)}.
$$
\end{enumerate}

\item \emph{Benedicks-Amrein-Berthier's uncertainty principle for $\ff_\pm$:}\\
 If $S,\,\Sigma$ are subsets of finite measure   $0<\nu(S),\nu(\Sigma)< \infty$, then
there exists a constant $C(S,\Sigma)$ such that for all $f\in L^2(\R^d, \mu)$,
\begin{equation*}
 \|f\|^2_{L^2(\R^d, \mu)}\le C(S,\Sigma)\Big(\|f\|^2_{L^2(S^c, \mu)} +\|\ff_\pm(f)\|^2_{L^2(\Sigma^c, \mu)} \Big).
\end{equation*}

\item \emph{Global uncertainty principle for $\ff_\pm$:}\\
For $s,\;\beta>0$, there exists a constant $C_{s,\beta}$ such that for all $f\in L^2(\R^d, \mu)$,
$$
    \big\||x|^{s}f\big\|^{\frac{2\beta}{s+\beta}}_{L^2(\R^d, \mu)} \; \big\||\xi|^{\beta}\ff_\pm(f)\big\|^{\frac{2s}{s+\beta}}_{L^2(\R^d, \mu)}
    \ge C_{s,\beta} \norm{f}^{2}_{L^2(\R^d, \mu)}.
$$
\end{enumerate}

\end{theorem}

\end{document}